\documentclass[12pt]{amsart}
\usepackage{amssymb}
\newtheorem{theorem}{Theorem}
\newtheorem{lemma}[theorem]{Lemma}

\newtheorem{corollary}[theorem]{Corollary}

\newtheorem{question}{Question}
\newtheorem{definition}[theorem]{Definition}

\title{Cardinal invariants of cellular-Lindel\"of spaces}

\author{Angelo Bella}

\address{
Department of Mathematics and Computer Science \\
University of Catania \\
Citt\'a universitaria\\  
viale A. Doria 6 \\
95125 Catania, Italy}

\email{bella@dmi.unict.it}

\author{Santi Spadaro}

\address{Department of Mathematics and Computer Science \\
University of Catania \\
Citt\'a universitaria\\  
viale A. Doria 6 \\
95125 Catania, Italy}

\email{santidspadaro@gmail.com}



\begin{document}
\maketitle

\begin{abstract} 
A space $X$ is said to be \emph{cellular-Lindel\"of} if for every cellular family $\mathcal{U}$ there is a Lindel\"of subspace $L$ of $X$ which meets every element of $\mathcal{U}$. Cellular-Lindel\"of spaces generalize both Lindel\"of spaces and spaces with the countable chain condition. Solving questions of Xuan and Song, we prove that every cellular-Lindel\"of monotonically normal space is Lindel\"of and that every cellular-Lindel\"of space with a regular $G_\delta$-diagonal has cardinality at most $2^\mathfrak{c}$. We also prove that every normal cellular-Lindel\"of first-countable space has cardinality at most continuum under $2^{<\mathfrak{c}}=\mathfrak{c}$ and that every normal cellular-Lindel\"of space with a $G_\delta$-diagonal of rank $2$ has cardinality at most continuum. \end{abstract}

\section{Introduction}

Two of the most important cardinal inequalities regarding Hausdorff topological spaces are Arhangel'skii's Theorem (see \cite{A} and \cite{H}, for example) stating that the cardinality of every Lindel\"of first-countable space does not exceed the continuum and the Hajnal-Juh\'asz inequality (see \cite{J}), which in the countable case says that every first-countable space with the countable chain condition has cardinality at most continuum. 

The weak Lindel\"of property is a common generalization of the Lindel\"of property and the countable chain condition which may be used to state a common strengthening of Arhangel'skii's Theorem and of the Hajnal-Juh\'asz inequality within the realm of normal spaces.

\begin{definition}
A space $X$ is weakly Lindel\"of if for every open cover $\mathcal{U}$ of $X$ there is a countable subcollection $\mathcal{V} \subset \mathcal{U}$ such that $X \subset \overline{\bigcup \mathcal{V}}$.
\end{definition}

It can be readily seen that every Lindel\"of space is weakly Lindel\"of. It's also easy to see that every open cover which does not have a countable subcollection with a dense union can be refined to an uncountable \emph{cellular family}, that is an uncountable family of pairwise disjoint non-empty open sets. So the countable chain condition implies the weak Lindel\"of property.

\begin{theorem}
(Bell, Ginsburg and Woods, \cite{BGW}) Every normal weakly Lindel\"of first-countable space has cardinality at most continuum.
\end{theorem}

The problem of whether normality can be relaxed to regularity in the above theorem is still open.

However, there are (see \cite{BGW}) Hausdorff non-regular examples of weakly Lindel\"of first-countable spaces of arbitrarily large cardinality, so one cannot expect to find a common generalization to Arhangel'skii's Theorem and the Hajnal-Juh\'asz inequality by using the weak Lindel\"of property.

In \cite{BS} we proposed another common generalization of the ccc and the Lindel\"of property. 

\begin{definition}
A space is called \emph{cellular-Lindel\"of} if for every cellular family $\mathcal{U}$ there is a Lindel\"of subspace $L$ of $X$ such that $U \cap L \neq \emptyset$, for every $U \in \mathcal{U}$.
\end{definition}

We noted that every cellular-Lindel\"of first-countable space has cardinality at most $2^\mathfrak{c}$ and asked whether the bound could be improved from $2^{\mathfrak{c}}$ to $\mathfrak{c}$. 

\begin{question} \cite{BS} \label{mainquestion}
Let $X$ be a first-countable cellular-Lindel\"of space. Is $|X| \leq \mathfrak{c}$?
\end{question}

A positive answer would lead to a common generalization of the Arhangel'skii and Hajnal-Juh\'asz inequalities.

In this note we prove that, under $2^{<\mathfrak{c}} = \mathfrak{c}$, every normal cellular-Lindel\"of first countable space has cardinality at most continuum, thus partially solving Question $\ref{mainquestion}$. We also prove, solving a question of Xuan and Song from \cite{XS2}, that every monotonically normal cellular-Lindel\"of space is Lindel\"of. As a byproduct, we obtain that Question $\ref{mainquestion}$ has a positive answer in ZFC for the class of monotonically normal spaces. We also prove that every normal cellular-Lindel\"of space with a rank $2$ diagonal has cardinality at most continuum. This gives a partial answer to a question from \cite{XS1}.

In our proofs we will sometimes use elementary submodels of the structure $(H(\mu), \epsilon)$. We believe that they make closing off arguments more transparent and concise. We encourage readers who have not done so already to acquaint themselves with Dow's survey \cite{D}. Recall that $H(\mu)$ is the set of all sets whose transitive closure has cardinality smaller than $\mu$. When $\mu$ is regular uncountable, $H(\mu)$ is known to satisfy all axioms of set theory, except the power set axiom. We say, informally, that a formula is satisfied by a set $S$ if it is true when all bounded quantifiers are restricted to $S$. A set $M \subset H(\mu)$ is said to be an elementary submodel of $H(\mu)$ (and we write $M \prec H(\mu)$) if a formula with parameters in $M$ is satisfied by $H(\mu)$ if and only if it is satisfied by $M$. 

The downward L\"owenheim-Skolem theorem guarantees that for every $S \subset H(\mu)$, there is an elementary submodel $M \prec H(\mu)$ such that $|M| \leq |S| \cdot \omega$ and $S \subset M$. This theorem is sufficient for many applications, but it is often useful (especially in cardinal bounds for topological spaces) to have the following closure property. We say that $M$ is $\kappa$-closed if for every $S \subset M$ such that $|S| \leq \kappa$ we have $S \in M$. For every countable set $S \subset H(\mu)$ there is always a $\kappa$-closed elementary submodel $M \prec H(\mu)$ such that $|M|=2^{\kappa}$ and $S \subset M$.

The following theorem is also used often: let $M \prec H(\mu)$ such that $\kappa + 1 \subset M$ and $S \in M$ be such that $|S| \leq \kappa$. Then $S \subset M$.

All spaces under consideration are assumed to be Hausdorff. Undefined notions can be found in \cite{E} for topology and \cite{Ku} for set theory. Our notation regarding cardinal functions mostly follows \cite{J}. 

\section{When is a cellular-Lindel\"of space Lindel\"of?}

It is apparent from the definition that every ccc space is cellular-Lindel\"of and every Lindel\"of space is cellular-Lindel\"of. The converses to either of the previous two implications do not hold as can be shown by simple examples distinguishing the ccc and the Lindel\"of property. 

Moreover, the weak Lindel\"of property does not imply the cellular-Lindel\"ofness. In \cite{XS1} Xuan and Song even provided an example of a weakly Lindel\"of Moore space which is not cellular-Lindel\"of. However, the question about the existence of a cellular-Lindel\"of non-weakly Lindel\"of space is still open.

We will prove that the cellular-Lindel\"of property and the Lindel\"ofness are equivalent for monotonically normal spaces. This solves Questions 4.11--4.13 from \cite{XS2}.

Given a topological space $X$ we indicate with $\mathcal{U}(X)$ the set of all pairs $(x,U)$, where $U$ is open and $x \in U$.

\begin{definition}
A topological space $(X, \tau)$ is called \emph{monotonically normal} if there exists an operator $H: \mathcal{U}(X) \to \tau$ with the following properties:

\begin{enumerate}
\item $x \in H(x,U) \subset U$, for every $(x, U) \in \mathcal{U}(X)$.
\item If $H(x,U) \cap H(y, V) \neq \emptyset$ then $x \in V$ or $y \in U$.
\end{enumerate}
\end{definition}

Monotonically normal spaces generalize both metric spaces and linearly ordered spaces. Moreover, monotone normality is a hereditary property, so even every \emph{GO space} (i.e., a subspace of a linearly ordered space) is monotonically normal (see \cite{HLZ} and \cite{Bo}).

The proof of Theorem $\ref{MNtheorem}$ is a variation on the proof of Theorem 3.17 from \cite{JTW}. We're going to need a characterization of paracompactness for monotonically normal spaces due to Balogh and Rudin.

\begin{lemma}
(Balogh and Rudin \cite{BR}) A monotonically normal space $X$ is paracompact if and only if it does not contain a closed subset homeomorphic to a stationary subset of a regular uncountable cardinal.
\end{lemma}

\begin{theorem} \label{MNtheorem}
Let $X$ be a monotonically normal cellular-Lindel\"of space. Then $X$ is Lindel\"of.
\end{theorem}

\begin{proof}
As noted by the authors of \cite{XS2}, the space $X$ has countable extent. Therefore, to show that $X$ is Lindel\"of, it is sufficient to prove that $X$ is paracompact. If the space $X$ were not paracompact, it would contain a closed copy $S$ of a stationary set of some regular uncountable cardinal $\kappa$. Let $D$ be the set of all isolated points of $S$. By the nature of the topology on $S$ we can choose, for every $x \in X$ a neighbourhood $U_x$ of $x$ such that $|U_x \cap D| < \kappa$.

For every $x \in D$, choose an open set $V_x$ such that $V_x \cap D=\{x\}$. Note that $\mathcal{V}=\{H(x, H(x, V_x)): x \in D \}$ is a cellular family, and hence there is a Lindel\"of subspace $L$ of $X$ such that $L \cap V \neq \emptyset$, for every $V \in \mathcal{V}$. Since $D$ has cardinality $\kappa$, the family $\mathcal{V}$ also has cardinality $\kappa$ and since $\kappa$ is a regular uncountable cardinal and $L$ is Lindel\"of, $L$ must contain a complete accumulation point for the family $\mathcal{V}$, that is a point $p \in L$ such that $\{V \in \mathcal{V}: O \cap V \neq \emptyset \}$ has cardinality $\kappa$, for every open neighbourhood $O$ of $p$. Clearly $p \notin H(x,V_x)$ for every $x \in D$, or otherwise $H(x,V_x)$ would be an open neighbourhood of $p$ which meets only one element of the family $\mathcal{V}$. So if $H(p, U_p) \cap H(x, H(x, V_x)) \neq \emptyset$, for some $x \in D$, we must have $x \in U_p$, by the second property of a monotone normality operator.

It follows that $\{x \in D: H(x,H(x, V_x)) \cap H(p, U_p) \neq \emptyset \} \subset D \cap U_p$, and since the latter set has cardinality smaller than $\kappa$, the point $p$ is not a complete accumulation point for the family $\mathcal{V}$, which is a contradiction.
\end{proof}

\begin{corollary}
Let $X$ be a GO space. Then $X$ is Lindel\"of if and only if $X$ is cellular-Lindel\"of.
\end{corollary}

\begin{corollary}
Let $X$ be a LOTS. Then $X$ is Lindel\"of if and only if $X$ is cellular-Lindel\"of.
\end{corollary}

\section{The cardinality of cellular-Lindel\"of first-countable spaces}

In this section we will establish a cardinal inequality for cellular-Lindel\"of spaces and find a class of spaces where cellular-Lindel\"of implies weakly Lindel\"of, under CH. The core of the argument of both results is the following lemma, which roughly says that a cellular-Lindel\"of space is close to being weakly Lindel\"of for covers of small size.

\begin{lemma} \label{mainlem}
Let $X$ be a cellular-Lindel\"of space. Let $\mathcal{U}$ be an open cover of $X$ having cardinality continuum. Then there is a subcollection $\mathcal{V} \subset \mathcal{U}$ of cardinality smaller than the continuum such that $X \subset \overline{\bigcup \mathcal{V}}$.
\end{lemma}

\begin{proof}
Suppose the statement is false and let $\{U_\alpha: \alpha < \mathfrak{c}\}$ be an enumeration of $\mathcal{U}$. Then we can find a strictly increasing sequence of ordinals $\{\alpha_\beta: \beta < \mathfrak{c}\}$ such that $V_\beta=U_{\alpha_{\beta+1}} \setminus \overline{\bigcup \{U_\gamma: \gamma \leq \alpha_\beta \}}$ is a non-empty open set, for every $\beta < \mathfrak{c}$. But then the cellular-Lindel\"of property implies the existence of a Lindel\"of subspace $L$ of $X$ such that $L \cap V_\beta \neq \emptyset$, for every $\beta < \mathfrak{c}$. Since $\mathcal{U}$ is an open cover of the Lindel\"of space $L$, there must be an ordinal $\beta < \mathfrak{c}$ such that $L \subset \bigcup \{U_\gamma: \gamma \leq \alpha_\beta\}$, but this contradicts $L \cap V_\beta \neq \emptyset$ and we are done.

\end{proof}

The following theorem gives a partial positive answer to Question 4 from \cite{BS}.

\begin{theorem} \label{thmineq}
($2^{<\mathfrak{c}}=\mathfrak{c}$) Let $X$ be a normal sequential cellular-Lindel\"of space such that $\chi(X) \leq \mathfrak{c}$. Then $|X| \leq \mathfrak{c}$
\end{theorem}

\begin{proof}
Let $M$ be an elementary submodel of $H(\theta)$, where $\theta$ is a regular large enough cardinal, such that $X \in M$, $\mathfrak{c}+1 \subset M$, $M$ is closed under sequence of cardinality less than $\mathfrak{c}$ and $|M|=\mathfrak{c}$. The sequentiality of $X$ implies that $X \cap M$ is a closed subspace of $X$.

We claim that $X \subset M$. Suppose by contradiction that there is a point $p \in X \setminus M$. By regularity we can find an open set $U \subset X$ such that $X \cap M \subset U$ and $p \notin \overline{U}$. Use $\chi(X) \leq \mathfrak{c}$ and $\mathfrak{c}+1 \subset M$ to choose, for every $x \in X \cap M$, an open neighbourhood $U_x$ of $x$ such that $U_x \subset U$ and $U_x \in M$. Let $V=\bigcup \{U_x: x \in X \cap M \}$ and note that $X \cap M$ and $X \setminus V$ are disjoint closed sets, so there are disjoint open sets $G_1$ and $G_2$ such that $X \cap M \subset G_1$ and $X \setminus V \subset G_2$. Now $\{U_x: x \in X \cap M \} \cup \{G_2\}$ is an open cover of $X$ having cardinality continuum, so by Lemma $\ref{mainlem}$ there must be a set $C \subset X \cap M$ of cardinality less than continuum such that $X \subset \overline{\bigcup \{U_x: x \in X \cap M \}} \cup \overline{G_2}$. But $\overline{G_2} \cap X \cap M = \emptyset$, so we actually have $X \cap M \subset \overline{\bigcup \{U_x: x \in C \}}$. But the fact that $M$ is closed under $<\mathfrak{c}$-sequences implies that $C \in M$, so:

$$M \models X \subset \overline{\bigcup \{U_x: x \in C \}}$$

Therefore, by elementarity, we can say that:

$$H(\theta) \models X \subset \overline{\bigcup \{U_x: x \in C \}}$$

But the above formula contradicts the fact that $p \notin \overline{U}$.

So $X \subset M$, which implies $|X| \leq |M| \leq 2^{\aleph_0}$.
\end{proof}

The following corollary is a consequence of Theorem $\ref{MNtheorem}$ and gives another partial positive answer to Question 4 from \cite{BS}.

\begin{corollary} \label{MNcorollary}
Every monotonically normal cellular-Lindel\"of first-countable space has cardinality at most continuum.
\end{corollary}

\begin{corollary} \label{WLcorollary}
(CH) Every normal first-countable cellular-Lindel\"of space is weakly Lindel\"of.
\end{corollary}

\begin{proof}
Let $X$ be a normal first-countable cellular-Lindel\"of space. By Theorem $\ref{thmineq}$, the space $X$ has cardinality at most $\mathfrak{c}=\omega_1$, so an argument similar to the one proving Lemma $\ref{mainlem}$ shows that $X$ is weakly Lindel\"of.
\end{proof}

\section{Cellular-Lindel\"of spaces with $G_\delta$-diagonals}

Recall that a space $X$ is said to have a \emph{$G_\delta$-diagonal} iff its diagonal is a countable intersection of open sets in $X^2$. This is equivalent to the existence of a sequence of open covers $\{\mathcal{U}_n: n < \omega \}$ of the space $X$ such that $\bigcap \{St(x, \mathcal{U}_n): n < \omega \}=\{x\}$, for every $x \in X$.

While Lindel\"of spaces with a $G_\delta$-diagonal have cardinality at most continuum, there is no bound on the cardinality of cellular-Lindel\"of spaces with a $G_\delta$-diagonal. Indeed, Shakmatov \cite{Sh} and Uspenskii \cite{Us} constructed examples of arbitrarily large ccc spaces with a $G_\delta$-diagonal.

Some cardinality restrictions can be obtained by using certain strengthenings of the notion of a $G_\delta$-diagonal. A space $X$ has a \emph{$G_\delta$-diagonal of rank 2} if there exists a sequence $\{\mathcal{U}_n: n < \omega \}$ of open covers of $X$ such that $\bigcap \{St(St(x,\mathcal{U}_n), \mathcal{U}_n): n < \omega \}=\{x\}$ for every $x \in X$.

In \cite{XS2}, Xuan and Song proved that every cellular-Lindel\"of space with a $G_\delta$-diagonal of rank 2 has cardinality at most $2^{\mathfrak{c}}$. We prove that in the presence of normality the bound can be improved.

\begin{theorem} \label{rank2theorem}
Let $X$ be a normal cellular-Lindel\"of space $X$ with a $G_\delta$-diagonal of rank 2. Then $|X| \leq \mathfrak{c}$.
\end{theorem}

\begin{proof}
Let $\{\mathcal{U}_n: n < \omega \}$ be a sequence of open covers witnessing that $X$ has a $G_\delta$-diagonal of rank 2 and suppose by contradiction that $|X|> \mathfrak{c}$. Set $F_n=\{\{x,y\} \in [X]^2: St(x, \mathcal{U}_n) \cap St(y, \mathcal{U}_n)=\emptyset \}$. Since $[X]^2=\bigcup \{F_n: n < \omega \}$, by the Erd\H{o}s-Rado theorem there is an uncountable set $S \subset X$ and an integer $n_0$ such that $[S]^2 \subset F_{n_0}$. Since $\mathcal{U}_{n_0}$ is an open cover of $X$, the set $S$ is closed. Therefore, we may pick an open set $V$ such that $S \subset V \subset \overline{V} \subset \bigcup \{St(x, \mathcal{U}_{n_0}): x \in S \}$. The family $\mathcal{U}=\{St(x, \mathcal{U}_{n_0}) \cap V: x \in S \}$ consists of pairwise disjoint non-empty open sets and thus there is a Lindel\"of subspace $L$ of $X$ which meets every member of $\mathcal{U}$. But then $\{U \cap L: U \in \mathcal{U}\}$ is an uncountable discrete family in $L$ and that is a contradiction.
\end{proof} 

One of the referees noted that the above theorem also follows from the main result of \cite{XS0}.

The following theorem answers Question 5.5 from \cite{XS2}.

\begin{theorem}
Every cellular-Lindel\"of space with a regular $G_\delta$-diagonal has cardinality bounded by $2^{\mathfrak{c}}$.
\end{theorem} 

\begin{proof}
Every Lindel\"of space with a $G_\delta$-diagonal has cardinality at most continuum, and hence every cellular-Lindel\"of space with a $G_\delta$-diagonal has cellularity at most continuum. Now, a space with cellularity at most continuum and a regular $G_\delta$-diagonal has cardinality at most $2^{\mathfrak{c}}$ by Theorem 4.2 of \cite{Go} (see also \cite{B}).
\end{proof}

\section{Open Problems}

The topic of cellular-Lindel\"of spaces is still in its infancy and several intriguing questions remain open about them. Besides Question 1, these are the main ones:

\begin{question}
Let $X$ be a cellular-Lindel\"of first-countable regular space. Is $|X| \leq \mathfrak{c}$?
\end{question}

Theorem $\ref{MNtheorem}$ and Corollary $\ref{MNcorollary}$ give partial positive answers to the above question.

\begin{question}
Let $X$ be a cellular-Lindel\"of first-countable normal space. Is $|X| \leq \mathfrak{c}$ in ZFC?
\end{question}

Theorem $\ref{MNtheorem}$ shows that the above question has a positive answer under $2^{<\mathfrak{c}}=\mathfrak{c}$.

\begin{question}
Is there a cellular-Lindel\"of non-weakly Lindel\"of (regular) space?
\end{question}

\begin{question}
Is every normal cellular-Lindel\"of first-countable space weakly Lindel\"of in ZFC?
\end{question}

Corollary $\ref{WLcorollary}$ shows that the above question has a positive answer under CH. 

If the following question had a positive answer, then, in view of Corollary $\ref{WLcorollary}$, the cellular-Lindel\"of and the weak Lindel\"of property would be equivalent for the class of normal first-countable spaces under CH.

\begin{question}
Is every normal first-countable weakly Lindel\"of space cellular-Lindel\"of under CH?
\end{question}

\begin{question}
\cite{XS2} Is there a normal weakly Lindel\"of non-cellular-Lindel\"of space?
\end{question}

\begin{question}
Let $X$ be a cellular-Lindel\"of regular space with a $G_\delta$-diagonal of rank 2. Is $|X| \leq \mathfrak{c}$?
\end{question}

By Theorem $\ref{rank2theorem}$ the above question has a positive answer for the class of normal spaces.

\begin{question}
Let $X$ be a cellular-Lindel\"of space with a regular $G_\delta$-diagonal. Is $|X| \leq \mathfrak{c}$?
\end{question}

\end{document}